\documentclass{cocv}
\usepackage{graphicx}
\usepackage{amssymb,stmaryrd}
\usepackage{subfigure}
\usepackage{cite}
\usepackage{url}
\usepackage{tikz}
\usepackage{bm}
\usepackage{amsmath}
\usepackage{marvosym}

\newtheorem{theorem}{Theorem}[section]
\newtheorem{lemma}{Lemma}[section]
\newtheorem{remark}{Remark}[section]
\newtheorem{example}{Example}[section]
\begin{document}
\title{Residual-type a posteriori error analysis of HDG methods for Neumann boundary control problems}
\thanks{This work is supported by State Key Program of National Science Foundation of China (11931003)
and National Nature Science Foundation of China (41974133, 11671157).}
%
\author{Haitao Leng}\address{School of Mathematical Sciences, South China Normal University, Guangzhou 510631, Guangdong, China.}

\author{Yanping Chen}\address{School of Mathematical Sciences, South China Normal University, Guangzhou 510631, Guangdong, China.}
%
\date{...}
\begin{abstract}
 We study a posteriori error analysis of
linear-quadratic boundary control problems under bilateral box constraints on the control which acts
through a Neumann type boundary condition.
We adopt the hybridizable discontinuous Galerkin method as discretization technique, and
the flux variables, the scalar variables and the boundary trace variables are all approximated by polynomials of degree k.
As for the control variable, it is discretized by the variational discretization concept.
Then an efficient and reliable a posteriori error estimator is introduced, and we prove that the error estimator provides
an upper bound and a lower bound for the error. Finally,
numerical results are presented to illustrate the performance of the obtained a posteriori error estimator.
\end{abstract}
%
%
\maketitle
\section{Introduction}
\label{intro}
Many optimization processes in science and engineering lead to optimal control problems where the sought state is a solution of
a partial differential equation. The complexity of such problem needs special care in order to obtain efficient numerical approximations
for the optimization problem. One particular method is adaptive finite element method, which can reduces the computational cost
and boosts the accuracy of the numerical solutions by locally refining the meshes around the singularity.

Although the adaptive finite element method has become a popular approach for numerical solutions of partial differential equations since
the work of Babu\v{s}ka and Rheinboldt \cite{br1978}, it has only quiet recently become popular for constrained optimal control problems.
The pioneer work concerning a posteriori error analysis for distributed optimal control problems is published by Liu and Yan \cite{ly2001}
for residual-type error estimators and Becker, Kapp, and Rannacher \cite{bkr2000} for goal-oriented error estimators.
Here, we further refer readers to \cite{hhik2008, llmt2002, yb2015, yk2014, zyy2014} for residual-type estimators and \cite{bv2009, hh2010}
for goal-oriented approach. Recently, in order to guarantee the performance of the a posteriori error estimator theoretically, many scholars
have tried
to prove the convergence of an adaptive finite element algorithm for distributed optimal control problems in \cite{gikh2007, gy2017, ksr2014, lc2018}.

Compared to distributed optimal control problems, there exists limited work on a posteriori error analysis for boundary optimal control problems.
In \cite{ly2001a}, the convex Neumann boundary control problem was considered on polygonal or Lipschitz piecewise $\mathcal{C}^2$ domain. Then
a residual-type a posteriori error estimator was introduced, and the authors proved that the estimator provided an upper bound for the errors in
the state and the control. In \cite{hiis2006}, by introducing a Lagrange multiplier, the authors derived
an efficient and reliable residual-type a posteriori error estimator for Neumann boundary control problems on polygonal domain.
In \cite{krs2014}, Kohls, R\"{o}sch and Siebert derived a unifying framework for the a posteriori error analysis of control constrained
linear-quadratic optimal control problems for the full and variational discretizations.
In \cite{by2017}, Benner and Y\"{u}cel investigated symmetric interior penalty Galerkin methods for Neumann boundary control problems with
an extra coefficient in cost functional. By invoking a Lagrange multiplier associated with the control constraints, an efficient and reliable
residual-type a posteriori error estimator was obtained for the errors in the state, adjoint, control and co-control. As for Dirichlet boundary
control problems, we just mention \cite{cgn2017, glty2018} and references therein for more details on a posteriori error analysis.

Recently, the hybridizable discontinuous Galerkin (HDG) methods \cite{cgl2009}, which keep the advantages of discontinuous Galerkin (DG) methods
and result in a system with significantly reduced degrees of freedom, have been proposed for convection diffusion problem \cite{fqz2015}, interface
problem \cite{cc2019}, flow problem\cite{npc2010}, optimal control problem \cite{chsszz2018, ghmszz2018}, and so on.
In \cite{cz2012, cz2013, cz2013a}, Cockburn and Zhang studied HDG methods for second
order elliptic problems, and an a posteriori error estimator with postprocessing solutions was obtained.
To the best of our knowledge, there exists no work on residual-type a posteriori error analysis of HDG methods for boundary control problems.

In this paper, we investigate a posteriori error analysis of Neumann optimal control problems under bilateral box constraints on the control.
The HDG method is used as discretization technique, and the flux variables, the scalar variables and the boundary trace variables are discretized
by polynomials of degree $k$. As for the control variable, we adopt the variational discretization concept proposed by Hinze in \cite{h2005} for
approximation. Then an efficient and reliable residual-type a posteriori error estimator without any postprocessing solutions is introduced,
and we prove that the error estimator provides not only an upper bound but also a lower bound up to data oscillations for the errors.
Finally, numerical experiments are presented to validate the performance of the obtained estimator.

The remainder of the paper is arranged as follows: In Section \ref{sec2} we introduce the model problem and the associated optimality system.
In Section \ref{sec3} the discrete optimality system is given, and we prove that the discrete scheme has a unique solution.
Then we prove the reliability and efficiency of the error estimator in Section \ref{sec4} and Section \ref{sec5} respectively. Numerical
experiments are presented in Section \ref{sec6} to validate the performance of the obtained estimator. Finally, some conclusions are provided
in Section \ref{sec7}.

Throughout this paper, let $C$ with or without subscript be a generic positive constant independent of the mesh size. For ease of exposition,
we denote $A\leq CB$ by $A\lesssim B$.

\section{The Neumann boundary control problem}
\label{sec2}
Let $\Omega\in\mathbb{R}^d~(d=2,3)$ be a polygonal $(d=2)$ or polyhedral $(d=3)$ domain with boundary $\partial\Omega$.
Before we introduce the model problem, let us summarize some notation.
For bounded and open set $D\in\mathbb{R}^d$ or $D\in\mathbb{R}^{d-1}$, we denote the usual Sobolev spaces by $W^{s,p}(D)$
with norm $\|\cdot\|_{s,p,D}$ and seminorm $|\cdot|_{s,p,D}$. The Hilbertian Sobolev spaces are abbreviated by
$H^s(D)=W^{s,2}(D)$ with norm $\|\cdot\|_{s,D}$ and seminorm $|\cdot|_{s,D}$. For $s=0$, $H^0(D)$ coincides with $L^2(D)$,
and the inner product is denoted by $(\cdot,\cdot)_{D}$ for $D\in \mathbb{R}^d$ and $\langle\cdot,\cdot\rangle_D$ for
$D\in\mathbb{R}^{d-1}$. Furthermore, we define $H(div,\Omega):=\{\textbf{v}\in (L^2(\Omega))^d:\nabla\cdot\textbf{v}\in L^2(\Omega)\}$.

Based on the domain $\Omega$, we consider the following Neumann boundary control problem
\begin{equation}\label{cost}
\min_{y\in H^1(\Omega),u\in U_{ad}}\mathcal{J}(y,u)=\frac{1}{2}\|y-y_d\|_{0,\Omega}^2+\frac{\alpha}{2}\|u\|_{0,\partial\Omega}^2,
\end{equation}
subject to the elliptic equations
\begin{subequations}\label{state}
\begin{align}
-\Delta y+y &=f\quad \rm{in}~\Omega,\label{state:1}\\
\nabla y\cdot\textbf{n}&=u+g\quad \rm{on}~\partial\Omega,\label{state:2}
\end{align}
\end{subequations}
where the regularization parameter $\alpha$ is a positive constant, $y_d\in L^2(\Omega)$, $f\in L^2(\Omega)$, $g\in L^2(\partial\Omega)$,
$\textbf{n}$ is the unit vector normal to the boundary $\partial\Omega$. The set $U_{ad}$ of
constraints is given by
\begin{equation*}
U_{ad}=\{v\in L^2(\partial\Omega): u_a\leq v\leq u_b~a.e.~x\in\partial\Omega\},
\end{equation*}
where $u_a$ and $u_b$ are assumed to be constant, and that $u_a<u_b$.

From \cite{l1971}, we know that the Neumann boundary control problem (\ref{cost})-(\ref{state}) admits a unique solution $(y,u)\in H^1(\Omega)\times
L^2(\partial\Omega)$, and there exists an adjoint-state $z\in H^1(\Omega)$ such that
\begin{subequations}\label{optimality}
\begin{align}
-\Delta y+ y&=f\quad \rm{in}~\Omega,\label{optimality:1}\\
\nabla y\cdot\textbf{n}&=u+g\quad \rm{on}~\partial\Omega,\label{optimality:2}\\
-\Delta z+z&=y-y_d\quad \rm{in}~\Omega,\label{optimality:3}\\
\nabla z\cdot\textbf{n}&=0\quad \rm{on}~\partial\Omega,\label{optimality:4}\\
\langle\alpha u+z,v-u\rangle_{\partial\Omega}&\geq 0\quad \forall v\in U_{ad}.\label{optimality:5}
\end{align}
\end{subequations}
Moreover, the variational inequality (\ref{optimality:5}) is equivalent to the projection formula
\begin{equation}
u=\Pi_{U_{ad}}\Big(-\frac{1}{\alpha}z|_{\partial\Omega}\Big),\label{equ}
\end{equation}
where $\Pi_{U_{ad}}$ is the $L^2$-projection onto $U_{ad}$. Then let $\textbf{p}=-\nabla y$ and $\textbf{q}=-\nabla z$,
the optimality system (\ref{optimality}) can be rewritten in a mixed form as follows:
\begin{subequations}\label{mixed}
\begin{align}
\textbf{p}+\nabla y&=0\quad \rm{in}~\Omega,\label{mixed:1}\\
\nabla\cdot\textbf{p}+y&=f\quad \rm{in}~\Omega,\label{mixed:2}\\
-\textbf{p}\cdot\textbf{n}&=u+g\quad \rm{on}~\partial\Omega,\label{mixed:3}\\
\textbf{q}+\nabla z&=0\quad \rm{in}~\Omega,\label{mixed:4}\\
\nabla\cdot\textbf{q}+z&=y-y_d\quad \rm{in}~\Omega,\label{mixed:5}\\
-\textbf{q}\cdot\textbf{n}&=0\quad \rm{on}~\partial\Omega,\label{mixed:6}\\
\langle\alpha u+z,v-u\rangle_{\partial\Omega}&\geq 0\quad \forall v\in U_{ad}.\label{mixed:7}
\end{align}
\end{subequations}
\section{The HDG discretization}
\label{sec3}
Let $\mathcal{T}_h$ be a conforming and shape regular partition of the domain $\Omega$. For each $K\in\mathcal{T}_h$, we denote $\partial K$ the set of its faces. Then we define $\partial\mathcal{T}_h=\{\partial K: K\in\mathcal{T}_h\}$. Denote $\mathcal{E}_h^o$ the set of all interior faces of $\mathcal{T}_h$
and $\mathcal{E}_h^{\partial}$ the set of all boundary faces of $\mathcal{T}_h$. Then we define $\mathcal{E}_h=\mathcal{E}_h^o\cup\mathcal{E}_h^{\partial}$.
For any $K\in\mathcal{T}_h$ and $F\in\mathcal{E}_h$, $h_K$ and $h_E$ denote the diameters of the element $K$ and the face $F$ respectively.
Furthermore, we define the mesh-dependent inner product by
\begin{equation*}
(w,v)_{\mathcal{T}_h}=\sum_{K\in\mathcal{T}_h}(w,v)_K,\quad \langle w,v\rangle_{\partial\mathcal{T}_h}=\sum_{K\in\mathcal{T}_h}\langle w,v\rangle_{\partial K}.
\end{equation*}
For vector-valued functions, the notations are similarly defined by the dot product.

Based on the partition $\mathcal{T}_h$, we define the discontinuous finite element spaces for the flux variables, the scalar variables and the boundary trace
variables as following
\begin{align*}
&\textbf{V}_h^k=\{\textbf{v}\in(L^2(\Omega))^d:\textbf{v}|_K\in(\mathcal{P}^k(K))^d,~\forall K\in\mathcal{T}_h\},\\
&W_h^k=\{w\in L^2(\Omega):w|_K\in\mathcal{P}^k(K),~\forall K\in\mathcal{T}_h\},\\
& M_h^k=\{\mu\in L^2(\mathcal{E}_h):\mu|_F\in\mathcal{P}^k(F),~\forall F\in\mathcal{E}_h\},
\end{align*}
where $\mathcal{P}^k(S)$ is the set of polynomials of degree no larger than $k$ on the domain $S$. In this paper, we adopt the variational concept
proposed by Hinze \cite{h2005} for the control variable, which suggests to approximate the state equation but not the control variable.
Therefore the control variable will be implicitly discretized by formula (\ref{equ}). Then the HDG scheme of the system (\ref{mixed}) reads as follows: Find
$(\textbf{p}_h,y_h,\widehat{y}_h)\in \textbf{V}_h^k\times W_h^k\times M_h^k$, $(\textbf{q}_h,z_h,\widehat{z}_h)\in
\textbf{V}_h^k\times W_h^k\times M_h^k$ and $u_h\in U_{ad}$ such that
\begin{subequations}\label{HDG}
\begin{align}
(\textbf{p}_h,\textbf{r}_1)_{\mathcal{T}_h}-(y_h,\nabla\cdot\textbf{r}_1)_{\mathcal{T}_h}+\langle\widehat{y}_h,\textbf{r}_1\cdot\textbf{n}\rangle_
{\partial\mathcal{T}_h}&=0,\label{HDG:1}\\
-(\textbf{p}_h,\nabla w_1)_{\mathcal{T}_h}+(y_h,w_1)_{\mathcal{T}_h}
+\langle \widehat{\textbf{p}}_h\cdot\textbf{n},w_1\rangle_{\partial\mathcal{T}_h}&=(f,w_1)_{\mathcal{T}_h},\label{HDG:2}\\
\langle\widehat{\textbf{p}}_h\cdot\textbf{n},\mu_1\rangle_{\partial\mathcal{T}_h\backslash\partial\Omega}&=0,\label{HDG:3}\\
-\langle\widehat{\textbf{p}}_h\cdot\textbf{n},\mu_1\rangle_{\partial\Omega}&=\langle u_h+g,\mu_1\rangle_{\partial\Omega},\label{HDG:4}\\
(\textbf{q}_h,\textbf{r}_2)_{\mathcal{T}_h}-(z_h,\nabla\cdot\textbf{r}_2)_{\mathcal{T}_h}+\langle\widehat{z}_h,\textbf{r}_2\cdot
\textbf{n}\rangle_{\partial\mathcal{T}_h}&=0,\label{HDG:5}\\
-(\textbf{q}_h,\nabla w_2)_{\mathcal{T}_h}+(z_h,w_2)_{\mathcal{T}_h}+\langle\widehat{\textbf{q}}_h\cdot\textbf{n},w_2\rangle
_{\partial\mathcal{T}_h}&=(y_h-y_d,w_2)_{\mathcal{T}_h},\label{HDG:6}\\
\langle\widehat{\textbf{q}}_h\cdot\textbf{n},\mu_2\rangle_{\partial\mathcal{T}_h\backslash\partial\Omega}&=0,\label{HDG:7}\\
-\langle\widehat{\textbf{q}}_h\cdot\textbf{n},\mu_2\rangle_{\partial\Omega}&=0,\label{HDG:8}\\
\langle\alpha u_h+\widehat{z}_h,v-u_h\rangle_{\partial\Omega}&\geq 0,\label{HDG:9}
\end{align}
\end{subequations}
for any $(\textbf{r}_1,w_1,\mu_1)\in\textbf{V}_h^k\times W_h^k\times M_h^k$, $(\textbf{r}_2,w_2,\mu_2)\in\textbf{V}_h^k\times W_h^k\times M_h^k$
and $v\in U_{ad}$. Similarly, we know that
the inequality (\ref{HDG:9})
 is equivalent to the following projection formula
\begin{equation*}
u_h=\Pi_{U_{ad}}\Big(-\frac{1}{\alpha}\widehat{z}_h|_{\partial\Omega}\Big).
\end{equation*}
Here the normal component of numerical fluxes $\widehat{\textbf{p}}_h\cdot\textbf{n}$ and
$\widehat{\textbf{q}}_h\cdot\textbf{n}$ is defined as
\begin{align*}
&\widehat{\textbf{p}}_h\cdot\textbf{n}=\textbf{p}\cdot\textbf{n}+\tau_1(y_h-\widehat{y}_h)\quad \rm{on}~\partial\mathcal{T}_h,\\
&\widehat{\textbf{q}}_h\cdot\textbf{n}=\textbf{q}\cdot\textbf{n}+\tau_2(z_h-\widehat{z}_h)\quad \rm{on}~\partial\mathcal{T}_h,
\end{align*}
for stabilization parameters $\tau_1$ and $\tau_2$.

For ease of exposition, we define operators $\mathcal{B}$ by
\begin{align*}
&\mathcal{B}(\textbf{r}_1,w_1,\mu_1;\textbf{r}_2,w_2,\mu_2;\tau)\\
=&(\textbf{r}_1,\textbf{r}_2)_{\mathcal{T}_h}-(w_1,\nabla\cdot\textbf{r}_2)_{\mathcal{T}_h}+\langle\mu_1,\textbf{r}_2\cdot\textbf{n}\rangle
_{\partial\mathcal{T}_h}\\
&+(\nabla\cdot\textbf{r}_1,w_2)_{\mathcal{T}_h}+(w_1,w_2)_{\mathcal{T}_h}\\
&+\langle\tau(w_1-\mu_1),w_2\rangle_{\partial\mathcal{T}_h}-\langle\textbf{r}_1\cdot\textbf{n}+\tau(w_1-\mu_1),\mu_2\rangle_{\partial\mathcal{T}_h},
\end{align*}
Then the HDG scheme (\ref{HDG}) can be rewritten according to the operator $\mathcal{B}$: Find
$(\textbf{p}_h,y_h,\widehat{y}_h)\in\textbf{V}_h^k\times W_h^k\times M_h^k$, $(\textbf{q}_h,z_h,\widehat{z}_h)\in
\textbf{V}_h^k\times W_h^k\times M_h^k$ and $u_h\in U_{ad}$ such that
\begin{subequations}\label{bbs}
\begin{align}
\mathcal{B}(\textbf{p}_h,y_h,\widehat{y}_h;\textbf{r}_1,w_1,\mu_1;\tau_1)&=(f,w_1)_{\mathcal{T}_h}+\langle u_h+g,\mu_1\rangle_{\partial\Omega},\label{bbs:1}\\
\mathcal{B}(\textbf{q}_h,z_h,\widehat{z}_h;\textbf{r}_2,w_2,\mu_2;\tau_2)&=(y_h-y_d,w_2)_{\mathcal{T}_h},\label{bbs:2}\\
\langle\alpha u_h+\widehat{z}_h,v-u_h\rangle_{\partial\Omega}&\geq 0,\label{bbs:3}
\end{align}
\end{subequations}
for any $(\textbf{r}_1,w_1,\mu_1)\in\textbf{V}_h^k\times W_h^k\times M_h^k$, $(\textbf{r}_2,w_2,\mu_2)\in\textbf{V}_h^k\times W_h^k\times M_h^k$
and $v\in U_{ad}$.
\begin{theorem}
We assume that $\tau_1=\tau_2>0$ on $\partial \mathcal{T}_h$ and $0\in U_{ad}$. Then the system (\ref{bbs}) has a unique solution.
\end{theorem}
\begin{proof}
Since the system (\ref{bbs}) is finite dimensional, we only need to prove that the system (\ref{bbs}) just has the zero solution for
the case of $f=y_d=g=0$. Let $(\textbf{r}_1,w_1,\mu_1)=(\textbf{q}_h,-z_h,-\widehat{z}_h)$ in (\ref{bbs:1}) and
$(\textbf{r}_2,w_2,\mu_2)=(-\textbf{p}_h,y_h,\widehat{y}_h)$ in (\ref{bbs:2}), we have
\begin{align*}
0=&\mathcal{B}(\textbf{p}_h,y_h,\widehat{y}_h;\textbf{q}_h,-z_h,-\widehat{z}_h;\tau_1)+
\mathcal{B}(\textbf{q}_h,z_h,\widehat{z}_h;-\textbf{p}_h,y_h,\widehat{y}_h;\tau_2)\\
=&(y_h,y_h)_{\mathcal{T}_h}-\langle u_h,\widehat{z}_h\rangle_{\partial\Omega}\geq (y_h,y_h)_{\mathcal{T}_h}
+\alpha\langle u_h,u_h\rangle_{\partial\Omega},
\end{align*}
from (\ref{bbs:3}) and the assumption $0\in U_{ad}$. Hence $y_h=0$ and $u_h=0$. Furthermore, let $(\textbf{r}_1,w_1,\mu_1)=
(\textbf{p}_h,y_h,\widehat{y}_h)$ in (\ref{bbs:1}) and $(\textbf{r}_2,w_2,\mu_2)=(\textbf{q}_h,z_h,\widehat{z}_h)$ in
(\ref{bbs:2}), we have
\begin{align*}
0=&(\textbf{p}_h,\textbf{p}_h)_{\mathcal{T}_h}+(y_h,y_h)_{\mathcal{T}_h}+\langle\tau_1(y_h-\widehat{y}_h),y_h-\widehat{y}_h\rangle_{\partial\mathcal{T}_h},\\
0=&(\textbf{q}_h,\textbf{q}_h)_{\mathcal{T}_h}+(z_h,z_h)_{\mathcal{T}_h}+\langle\tau_2(z_h-\widehat{z}_h),z_h-\widehat{z}_h\rangle_{\partial\mathcal{T}_h}
\end{align*}
Therefore $\textbf{p}_h=0$, $\widehat{y}_h=0$, $\textbf{q}_h=0$, $z_h=0$ and $\widehat{z}_h=0$. Then we conclude the proof.
\end{proof}

\section{The residual-type a posteriori error estimator}
\label{sec4}
\subsection{Auxiliary results}
Before we start to prove a posteriori error estimator for the model problem, we first provide some auxiliary results that will play an important
role in the proof.

For each element $K\in\mathcal{T}_h$ and face $F\in\mathcal{E}_h$, we denote $\Pi_j^o$ and $\Pi_j^{\partial}$ the $L^2$-projections onto
$\mathcal{P}^j(K)$ and $\mathcal{P}^j(F)$ for the nonnegative integer $j$. Then, from \cite{cc2019} we have the following error estimates
\begin{lemma}\label{lem1}
For any $K\in\mathcal{T}_h$ and $F\in\mathcal{E}_h$, we have
\begin{align*}
\|\Pi_j^o v\|_{0,K}\leq& \|v\|_{0,K}\quad \forall v\in L^2(K),\\
\|\Pi_j^{\partial}v\|_{0,F}\leq&\|v\|_{0,F}\quad \forall v\in L^2(F),\\
\|v-\Pi_0^o v\|_{0,K}\lesssim& h_K\|\nabla v\|_{0,K}\quad \forall v\in H^1(K),\\
\|v-\Pi_0^o v\|_{0,\partial K}\lesssim& h_K^{1/2}\|\nabla v\|_{0,K}\quad \forall v\in H^1(K).
\end{align*}
\end{lemma}

We conclude this subsection by introducing a lemma that has been proved in \cite{chz2017}.
\begin{lemma}\label{lem2}
Let $F$ be a face of the element $K\in\mathcal{T}_h$, $\textbf{n}_F$ the unit vector normal to $F$, and $s>0$. Assume that
$v$ is a given function in $H^{1+s}(K)$ and $\Delta v\in L^2(K)$. For any $w_h\in\mathcal{P}^k(F)$, we have
\begin{equation*}
\langle \nabla v\cdot\textbf{n}_F,w_h\rangle_F\lesssim h_F^{-1/2}\|w_h\|_{0,F}(\|\nabla v\|_{0,K}+h_K\|\Delta v\|_{0,K}).
\end{equation*}
\end{lemma}
\subsection{Reliability of the error estimator}
\label{sec4:2}
We begin this section by defining error estimators for each $K\in\mathcal{T}_h$ in the following
\begin{align*}
\eta_{s,K,1}=&\|\textbf{p}_h+\nabla y_h\|_{0,K},\quad\quad \eta_{as,K,1}=\|\textbf{q}_h+\nabla z_h\|_{0,K},\\
\eta_{s,K,2}=&h_K\|f-\nabla\cdot\textbf{p}_h-y_h\|_{0,K},\quad\eta_{as,K,2}=h_K\|y_h-y_d-\nabla\cdot\textbf{q}_h-z_h\|_{0,K},\\
\eta_{s,\partial K}=&h_K^{-1/2}\|y_h-\widehat{y}_h\|_{0,\partial K},\quad\quad \eta_{as,\partial K}=
h_K^{-1/2}\|z_h-\widehat{z}_h\|_{0,\partial K}.
\end{align*}
Furthermore, we define
\begin{align*}
&\eta_s^2=\sum_{K\in\mathcal{T}_h}\{\eta_{s,K,1}^2+\eta_{s,K,2}^2+\eta_{s,\partial K}^2\},\\
&\eta_{as}^2=\sum_{K\in\mathcal{T}_h}\{\eta_{as,K,1}^2+\eta_{as,K,2}^2+\eta_{as,\partial K}^2\}.
\end{align*}

Next, we consider the following auxiliary problem: Find $\textbf{p}(u_h),\textbf{q}(u_h)\in H(div,\Omega)$ and $y(u_h),z(u_h)\in H^1(\Omega)$ such that
\begin{subequations}\label{auxi}
\begin{align}
\textbf{p}(u_h)+\nabla y(u_h)&=0\quad \rm{in}~\Omega,\label{auxi:1}\\
\nabla\cdot\textbf{p}(u_h)+y(u_h)&=f\quad \rm{in}~\Omega,\label{auxi:2}\\
-\textbf{p}(u_h)\cdot\textbf{n}&=u_h+g\quad \rm{on}~\partial\Omega,\label{auxi:3}\\
\textbf{q}(u_h)+\nabla z(u_h)&=0\quad \rm{in}~\Omega,\label{auxi:4}\\
\nabla\cdot\textbf{q}(u_h)+z(u_h)&=y(u_h)-y_d\quad \rm{in}~\Omega,\label{auxi:5}\\
-\textbf{q}(u_h)\cdot\textbf{n}&=0\quad \rm{on}~\Omega.\label{auxi:6}
\end{align}
\end{subequations}

Now the error $\|u-u_h\|_{0,\partial\Omega}+\|\textbf{p}-\textbf{p}(u_h)\|_{0,\Omega}+\|y-y(u_h)\|_{0,\Omega}+\|\textbf{q}-\textbf{q}(u_h)\|_{0,\Omega}
+\|z-z(u_h)\|_{0,\Omega}$ can be bounded by $\|z_h-z(u_h)\|_{1,\Omega}$ and $\Big\{\sum_{K\in\mathcal{T}_h}\eta_{as,\partial K}^2\Big\}^{1/2}$.
\begin{lemma}\label{lem3}
Let $(u,\textbf{p},y,\textbf{q},z)$ and $(u_h,\textbf{p}_h,y_h,\widehat{y}_h,\textbf{q}_h,z_h,\widehat{z}_h)$ be the solutions of problems
(\ref{mixed}) and (\ref{bbs}) respectively. Moreover let $(\textbf{p}(u_h),y(u_h),\textbf{q}(u_h),z(u_h))$ as defined above. Then the following
error estimate holds
\begin{align*}
&\|u-u_h\|_{0,\partial\Omega}+\|\textbf{p}-\textbf{p}(u_h)\|_{0,\Omega}+\|y-y(u_h)\|_{0,\Omega}+
\|\textbf{q}-\textbf{q}(u_h)\|_{0,\Omega}\\
&+\|z-z(u_h)\|_{0,\Omega}\lesssim\Big\{\sum_{K\in\mathcal{T}_h}\eta_{as,\partial K}^2\Big\}^{1/2}+\|z_h-z(u_h)\|_{1,\Omega}
\end{align*}
\end{lemma}
\begin{proof}
From (\ref{mixed}), (\ref{auxi}) and integration by parts to yield
\begin{equation}
\langle z-z(u_h),u-u_h\rangle_{\partial\Omega}=\|y-y(u_h)\|_{0,\Omega}^2.\label{lem3-proof:1}
\end{equation}
Obviously, $(\textbf{p}-\textbf{p}(u_h),y-y(u_h))$ is the solution of system (\ref{mixed:1})-
(\ref{mixed:3}) with $g=-u_h$ and $f=0$, and $(\textbf{q}-\textbf{q}(u_h),z-z(u_h))$ is the solution of system
(\ref{mixed:4})-(\ref{mixed:6}) with $y_d=y(u_h)$. Therefore we have
\begin{align}
\|\textbf{p}-\textbf{p}(u_h)\|_{0,\Omega}+\|y-y(u_h)\|_{0,\Omega}&\lesssim\|u-u_h\|_{0,\partial\Omega},\label{lem3-proof:2}\\
\|\textbf{q}-\textbf{q}(u_h)\|_{0,\Omega}+\|z-z(u_h)\|_{0,\Omega}&\lesssim\|y-y(u_h)\|_{0,\Omega},\label{lem3-proof:3}
\end{align}
by the trace theorem. From (\ref{mixed:7}), (\ref{bbs:3}) and (\ref{lem3-proof:1}), we obtain
\begin{align}
\alpha\|u-u_h\|_{0,\partial\Omega}^2\leq&\langle \widehat{z}_h-z(u_h),u-u_h\rangle_{\partial\Omega}\nonumber\\
\leq&\|z_h-z(u_h)\|_{1,\Omega}\|u-u_h\|_{0,\partial\Omega}\nonumber\\
&+\langle z_h-\widehat{z}_h,(\textbf{p}-\textbf{p}(u_h))\cdot\textbf{n}\rangle_{\partial\Omega}\nonumber\\
\lesssim&\|z_h-z(u_h)\|_{1,\Omega}\|u-u_h\|_{0,\partial\Omega}\label{lem3-proof:4}\\
&+\Big\{\sum_{K\in\mathcal{T}_h}\eta_{as,\partial K}^2\Big\}^{1/2}\Big(\|\textbf{p}-\textbf{p}(u_h)\|_{0,\Omega}\nonumber\\
&+\|y-y(u_h)\|_{0,\Omega}\Big),\nonumber
\end{align}
by the trace theorem and Lemma \ref{lem2}. Then we can conclude the proof by combining (\ref{lem3-proof:2})-(\ref{lem3-proof:4}).
\end{proof}
\begin{lemma}\label{lem4}
Let $(\textbf{p}(u_h),y(u_h),\textbf{q}(u_h),z(u_h))$ and $(u_h,\textbf{p}_h,y_h,\widehat{y}_h,\textbf{q}_h,z_h,\widehat{z}_h)$
be the solutions of problems (\ref{auxi}) and (\ref{bbs}), then the following error estimates hold
\begin{align}
\|\nabla(y_h-y(u_h))\|_{0,K}\leq \eta_{s,K,1}+\|\textbf{p}_h-\textbf{p}(u_h)\|_{0,K},\label{lem4:1}\\
\|\nabla(z_h-z(u_h))\|_{0,K}\leq \eta_{as,K,1}+\|\textbf{q}_h-\textbf{q}(u_h)\|_{0,K},\label{lem4:2}
\end{align}
for each $K\in\mathcal{T}_h$.
\end{lemma}
\begin{proof}
Since $\textbf{p}(u_h)=-\nabla y(u_h)$ and $\textbf{q}(u_h)=-\nabla z(u_h)$ in each $K\in\mathcal{T}_h$, we can obtain the error estimates
(\ref{lem4:1}) and (\ref{lem4:2}) directly by the triangle inequality.
\end{proof}

Now we are ready to prove a posteriori error estimators for $\|\textbf{p}_h-\textbf{p}(u_h)\|_{0,\Omega}+\|y_h-y(u_h)\|_{0,\Omega}$
and $\|\textbf{q}_h-\textbf{q}(u_h)\|_{0,\Omega}+\|z_h-z(u_h)\|_{0,\Omega}$.
\begin{lemma}\label{lem5}
Let $(\textbf{p}(u_h),y(u_h),\textbf{q}(u_h),z(u_h))$ and $(u_h,\textbf{p}_h,y_h,\widehat{y}_h,\textbf{q}_h,z_h,\widehat{z}_h)$
be the solutions of problems (\ref{auxi}) and (\ref{bbs}), then we have
\begin{align}
&\|\textbf{p}_h-\textbf{p}(u_h)\|_{0,\Omega}+\|y_h-y(u_h)\|_{0,\Omega}+\|\tau_1^{1/2}(y_h-\widehat{y}_h)\|_{0,\partial\mathcal{T}_h}\lesssim\eta_s,\label{lem5:1}\\
&\|\textbf{q}_h-\textbf{q}(u_h)\|_{0,\Omega}+\|z_h-z(u_h)\|_{0,\Omega}+\|\tau_2^{1/2}(z_h-\widehat{z}_h)\|_{0,\partial\mathcal{T}_h}\lesssim\eta_s+\eta_{as},
\label{lem5:2}
\end{align}
for $\tau_1=\tau_2=h_K^{-1}$ on each $\partial K$ for all $K\in\mathcal{T}_h$, where
\begin{equation*}
\|\cdot\|_{0,\partial\mathcal{T}_h}^2=\sum_{K\in
\mathcal{T}_h}\|\cdot\|_{0,\partial K}^2.
\end{equation*}
\end{lemma}
\begin{proof}
According to the definition of the operator $\mathcal{B}$ to infer that
\begin{equation*}
\mathcal{B}(\textbf{p}(u_h)-\textbf{p}_h,y(u_h)-y_h,y(u_h)-\widehat{y}_h;\textbf{r},w,\mu;\tau_1)=0,
\end{equation*}
for any $(\textbf{r},w,\mu)\in\textbf{V}_h^k\times W_h^k\times M_h^k$. Then from the above equality and the definition of the operator $\mathcal{B}$,
we have
\begin{align}
&\|\textbf{p}_h-\textbf{p}(u_h)\|_{0,\Omega}^2+\|y_h-y(u_h)\|_{0,\Omega}^2+\|\tau_1^{1/2}(y_h-\widehat{y}_h)\|_{0,\partial\mathcal{T}_h}^2\nonumber\\
=&\mathcal{B}(\textbf{p}(u_h)-\textbf{p}_h,y(u_h)-y_h,y(u_h)-\widehat{y}_h;\textbf{p}(u_h)-\textbf{p}_h,y(u_h)-y_h,y(u_h)-\widehat{y}_h;\tau_1)\nonumber\\
=&\mathcal{B}(\textbf{p}(u_h)-\textbf{p}_h,y(u_h)-y_h,y(u_h)-\widehat{y}_h;\delta_{\textbf{p}},\delta_{y},\delta_{\widehat{y}};\tau_1),\nonumber
\end{align}
where $\delta_{\textbf{p}}=\textbf{p}(u_h)-\textbf{p}_h-\textbf{r}$, $\delta_{y}=y(u_h)-y_h-w$, and $\delta_{\widehat{y}}
=y(u_h)-\widehat{y}_h-\mu$ for any $(\textbf{r},w,\mu)\in\textbf{V}_h^k\times W_h^k\times M_h^k$. By integration by parts we yield
\begin{align}
&\mathcal{B}(\textbf{p}(u_h)-\textbf{p}_h,y(u_h)-y_h,y(u_h)-\widehat{y}_h;\delta_{\textbf{p}},\delta_{y},\delta_{\widehat{y}};\tau_1)\nonumber\\
=&-(\textbf{p}_h+\nabla y_h,\delta_{\textbf{p}})_{\mathcal{T}_h}+(f-\nabla\cdot\textbf{p}_h-y_h,\delta_y)_{\mathcal{T}_h}\nonumber\\
&+\langle y_h-\widehat{y}_h,\delta_{\textbf{p}}\cdot\textbf{n}\rangle_{\partial\mathcal{T}_h}+\langle \tau_1(\widehat{y}_h-y_h),\delta_y\rangle_{\partial\mathcal{T}_h}\nonumber\\
&-\langle(\textbf{p}(u_h)-\textbf{p}_h)\cdot\textbf{n}+\tau_1(\widehat{y}_h-y_h),\delta_{\widehat{y}}\rangle_{\partial\mathcal{T}_h}.\nonumber
\end{align}
From (\ref{HDG:3}), (\ref{HDG:4}) and (\ref{auxi:3}), we arrive at
\begin{align}
-\langle(\textbf{p}(u_h)-\textbf{p}_h)\cdot\textbf{n}+\tau_1(\widehat{y}_h-y_h),\delta_{\widehat{y}}\rangle_{\partial\mathcal{T}_h}=0.\nonumber
\end{align}
Now we set $\textbf{r}=\Pi_0^o(\textbf{p}(u_h)-\textbf{p}_h)$ in the definition of $\delta_{\textbf{p}}$ and
$w=\Pi_0^o(y(u_h)-y_h)$ in the definition of $\delta_{y}$. Then from Lemma \ref{lem1} we have
\begin{align}
&-(\textbf{p}_h+\nabla y_h,\delta_{\textbf{p}})_{\mathcal{T}_h}+(f-\nabla\cdot\textbf{p}_h-y_h,\delta_y)_{\mathcal{T}_h}\nonumber\\
\lesssim&\Big\{\sum_{K\in\mathcal{T}_h}\eta_{s,K,1}^2\Big\}^{1/2}\|\textbf{p}(u_h)-\textbf{p}_h\|_{0,\Omega}\nonumber\\
&+\Big\{\sum_{K\in\mathcal{T}_h}\eta_{s,K,2}^2\Big\}^{1/2}\|\nabla(y(u_h)-y_h)\|_{0,\Omega},\nonumber
\end{align}
and
\begin{align}
&\langle \tau_1(\widehat{y}_h-y_h),\delta_y\rangle_{\partial\mathcal{T}_h}\nonumber\\
\lesssim&\Big\{\sum_{K\in\mathcal{T}_h}\eta_{s,\partial K}^2\Big\}^{1/2}\|\nabla(y(u_h)-y_h)\|_{0,\Omega}.\nonumber
\end{align}
By using Lemma \ref{lem2} to yield
\begin{align}
&\langle y_h-\widehat{y}_h,\delta_{\textbf{p}}\cdot\textbf{n}\rangle_{\partial\mathcal{T}_h}\nonumber\\
=&\sum_{K\in\mathcal{T}_h}\sum_{F\in\partial K}\langle y_h-\widehat{y}_h,\delta_{\textbf{p}}\cdot\textbf{n}\rangle_F\nonumber\\
\lesssim&\Big\{\sum_{K\in\mathcal{T}_h}\eta_{s,\partial K}^2\Big\}^{1/2}\|\textbf{p}(u_h)-\textbf{p}_h\|_{0,\Omega}\nonumber\\
&+\Big\{\sum_{K\in\mathcal{T}_h}\eta_{s,\partial K}^2\Big\}^{1/2}\Big(\Big\{\sum_{K\in\mathcal{T}_h}\eta_{s,K,2}^2\Big\}^{1/2}+
\|y_h-y(u_h)\|_{0,\Omega}\Big).\nonumber
\end{align}
Now we can obtain the approximation result (\ref{lem5:1}) by combining Lemma \ref{lem4}, Young's inequality and the above equalities and inequalities.
Moreover, the error estimate (\ref{lem5:2}) can be proved similarly.
\end{proof}
\begin{remark}
Compared to the error estimators introduced in \cite{cz2012, cz2013, cz2013a}, the Lemma \ref{lem5} provides an a posteriori error
estimator without any postprocessing solutions, hence it is easer to calculate.
\end{remark}

Combining Lemma \ref{lem3}, Lemma \ref{lem4} and Lemma \ref{lem5} results in the following reliability estimate
\begin{theorem}
Let $(u,\textbf{p},y,\textbf{q},z)$ and $(u_h,\textbf{p}_h,y_h,\widehat{y}_h,\textbf{q}_h,z_h,\widehat{z}_h)$ be the solutions of problems
(\ref{mixed}) and (\ref{bbs}). Then we have the following error estimate
\begin{align*}
&\|u-u_h\|_{0,\partial\Omega}+\|\textbf{p}-\textbf{p}_h\|_{0,\Omega}+\|y-y_h\|_{1,\Omega}\\
&+\|\textbf{q}-\textbf{q}_h\|_{0,\Omega}+\|z-z_h\|_{1,\Omega}\lesssim\eta_s+\eta_{as},
\end{align*}
for $\tau_1=\tau_2=h_K^{-1}$ on each $\partial K$ for all $K\in\mathcal{T}_h$.
\end{theorem}
\section{Efficiency of the error estimator}
\label{sec5}
In this section, we will prove that, up to data oscillations, the estimator $\eta_s+\eta_{as}$ also provides a lower bound for the error.
Especially, we will show that the local contributions of the estimator can be bounded from above by the local constituents of the error and
the associated data oscillations.

First of all, we define the data oscillations by
\begin{align*}
osc^2(f,\mathcal{T}_h)&=\sum_{K\in\mathcal{T}_h}osc^2(f,K),\\
osc^2(y_d,\mathcal{T}_h)&=\sum_{K\in\mathcal{T}_h}osc^2(y_d,K),
\end{align*}
where
\begin{align*}
osc(f,K)&=h_K\|f-\Pi_k^of\|_{0,K},\\
osc(y_d,K)&=h_K\|y_d-\Pi_k^oy_d\|_{0,K}.
\end{align*}
Obviously, $osc(f,K)$ and $osc(y_d,K)$ are of same order with $\eta_{s,K,2}$ and $\eta_{as,K,2}$ for non smooth
$f$ and $y_d$ and of higher order for smooth $f$ and $y_d$.

Next we denote by $\lambda_i^K$, $1\leq i\leq 3$, the barycentric coordinates of $K\in\mathcal{T}_h$ and refer to $E_K=27\Pi_{i=1}^3\lambda_i^K$
as the associated element bubble function. From \cite{hiis2006}, we have
\begin{subequations}\label{bubble}
\begin{align}
&\|p_K\|_{0,K}^2\lesssim(p_K,p_KE_K)_{K}\quad K\in\mathcal{T}_h,\label{bubble:1}\\
&\|p_KE_K\|_{0,K}\lesssim\|p_K\|_{0,K}\quad K\in\mathcal{T}_h,\label{bubble:2}\\
&\|p_KE_K\|_{1,K}\lesssim h_K^{-1}\|p_K\|_{0,K}\quad K\in\mathcal{T}_h,\label{bubble:3}
\end{align}
\end{subequations}
for $p_K\in\mathcal{P}^k(K)$.
Then the following error estimates hold.
\begin{theorem}\label{lem6}
Let $(u,\textbf{p},y,\textbf{q},z)$, $(\textbf{p}(u_h),y(u_h),\textbf{q}(u_h),z(u_h))$ and $(u_h,\textbf{p}_h,y_h,\widehat{y}_h,\\
\textbf{q}_h,z_h,\widehat{z}_h)$ be the solutions of problems (\ref{mixed}), (\ref{auxi}) and (\ref{bbs}) respectively. Then we have
\begin{align}
\eta_{s,K,1}\leq&\|\textbf{p}-\textbf{p}_h\|_{0,K}+\|\nabla(y-y_h)\|_{0,K}\quad K\in\mathcal{T}_h,\label{lem6:1}\\
\eta_{s,K,2}\lesssim& osc(f,K)+\|\textbf{p}-\textbf{p}_h\|_{0,K}+\|y-y_h\|_{0,K}\quad K\in\mathcal{T}_h,\label{lem6:2}\\
\sum_{K\in\mathcal{T}_h}\eta_{s,\partial K}^2\lesssim &\|u-u_h\|_{0,\partial\Omega}^2+\|\textbf{p}-\textbf{p}_h\|_{0,\Omega}^2+
\|y-y_h\|_{1,\Omega}^2+osc^2(f,\mathcal{T}_h),\label{lem6:3}
\end{align}
and
\begin{align}
\eta_{as,K,1}\leq&\|\textbf{q}-\textbf{q}_h\|_{0,K}+\|\nabla(z-z_h)\|_{0,K}\quad K\in\mathcal{T}_h,\label{lem6:4}\\
\eta_{as,K,2}\lesssim&osc(y_d,K)+\|\textbf{q}-\textbf{q}_h\|_{0,K}+\|z-z_h\|_{0,K}\nonumber\\
&+\|y-y_h\|_{0,K}\quad K\in\mathcal{T}_h,\label{lem6:5}\\
\sum_{K\in\mathcal{T}_h}\eta_{as,\partial K}^2\lesssim&\|u-u_h\|_{0,\partial\Omega}^2+\|\textbf{q}-\textbf{q}_h\|_{0,\Omega}^2+\|z-z_h\|_{1,\Omega}^2\nonumber\\
&+\|y-y_h\|_{0,\Omega}^2+osc^2(y_d,\mathcal{T}_h).\label{lem6:6}
\end{align}
\end{theorem}
\begin{proof}
Obviously, the inequalities (\ref{lem6:1}) and (\ref{lem6:4}) can be obtained directly by the triangle inequality. According to
the definition of $\eta_{s,K,2}$ and the triangle inequality we know that
\begin{equation*}
\eta_{s,K,2}\leq osc(f,K)+h_K\|\Pi_k^of-\nabla\cdot\textbf{p}_h-y_h\|_{0,K}.
\end{equation*}
Setting $p_K=\Pi_k^of-\nabla\cdot\textbf{p}_h-y_h$ we obtain
\begin{align*}
h_K^2\|p_K\|_{0,K}^2\lesssim& h_K^2(\Pi_k^of-f+f-\nabla\cdot\textbf{p}_h-y_h,E_Kp_K)_K\\
=&h_K^2(\Pi_k^of-f,E_Kp_K)_K+h_K^2(\nabla\cdot(\textbf{p}-\textbf{p}_h)+y-y_h,E_Kp_K)_K,
\end{align*}
from (\ref{bubble:1}). Then we can obtain the error estimate (\ref{lem6:2}) by using (\ref{bubble}), Young's inequality and the above two inequalities.
And the approximation result (\ref{lem6:5}) can be proved similarly. Now we turn to prove the error estimate (\ref{lem6:3}). From the proof of Lemma
\ref{lem5}, we have
\begin{align*}
\sum_{K\in\mathcal{T}_h}\eta_{s,\partial K}^2\lesssim&\|\textbf{p}(u_h)-\textbf{p}_h\|_{0,\Omega}^2+\|y(u_h)-y_h\|_{1,\Omega}^2\\
&+\sum_{K\in\mathcal{T}_h}\{\eta_{s,K,1}^2+\eta_{s,K,2}^2\}.
\end{align*}
Then by using (\ref{lem3-proof:2}), (\ref{lem6:1}), (\ref{lem6:2}), the triangle inequality and the above inequality to infer that
\begin{align*}
\sum_{K\in\mathcal{T}_h}\eta_{s,\partial K}^2\lesssim\|u-u_h\|_{0,\partial\Omega}+\|\textbf{p}-\textbf{p}_h\|_{0,\Omega}^2+
\|y-y_h\|_{1,\Omega}^2+osc^2(f,\mathcal{T}_h).
\end{align*}
Therefore the approximation result (\ref{lem6:3}) is derived. And the inequality (\ref{lem6:6}) can be proved similarly.
\end{proof}

\section{Numerical experiments}
\label{sec6}
Now we provide two examples in order to examine the quality of the derived estimator.
As we know, an adaptive algorithm consists of the loops "\textbf{SOLVE}$\rightarrow$\textbf{ESTIMATE}$\rightarrow$\textbf{MARK}$\rightarrow$\textbf{REFINE}". In this section, a fix-point iteration algorithm presented in \cite{zilz2015} is used for solving the model problem. In step \textbf{REFINE}, the newest vertex bisection algorithm \cite{s2007} is employed, and the following marking strategy is used in step \textbf{MARK}
\begin{equation*}
\sum_{K\in\mathcal{M}_h}\eta_K^2\geq\theta\eta^2,
\end{equation*}
where
\begin{align*}
\eta^2=\eta_s^2&+\eta_{as}^2,\\
\eta_{K}^2=\eta_{s,K,1}^2+\eta_{s,K,2}^2+\eta_{s,\partial K}^2&+\eta_{as,K,1}^2+\eta_{as,K,2}^2+\eta_{as,\partial K}^2.
\end{align*}
Furthermore, we define
\begin{align*}
E=&\|u-u_h\|_{0,\partial\Omega}+\|\textbf{p}-\textbf{p}_h\|_{0,\Omega}+\|y-y_h\|_{1,\Omega}\\
&+\|\textbf{q}-\textbf{q}_h\|_{0,\Omega}+\|z-z_h\|_{1,\Omega}.
\end{align*}
Here, we note that the figures of convergence history are plotted in log-log coordinates.
\begin{example}
Based on the domain $\Omega=(0,1)^2$, we consider an example with $u_a=-0.1$, $u_b=0.1$
and $\alpha=1$. Let the functions $f$, $y_d$ and $g$ be such that the Neumann boundary control problem
has the following exact solutions
\begin{align*}
y=\sin(2\pi x_1)\sin(2\pi x_2),\quad z=\cos(2\pi x_1),\quad u=\Pi_{U_{ad}}\Big\{-\frac{1}{\alpha}z|_{\partial\Omega}\Big\}.
\end{align*}
\end{example}

\begin{figure}[htbp]
\begin{center}
\includegraphics[width=6.5cm,height=5.4cm,clip,trim=0cm 0cm 1cm 1cm]{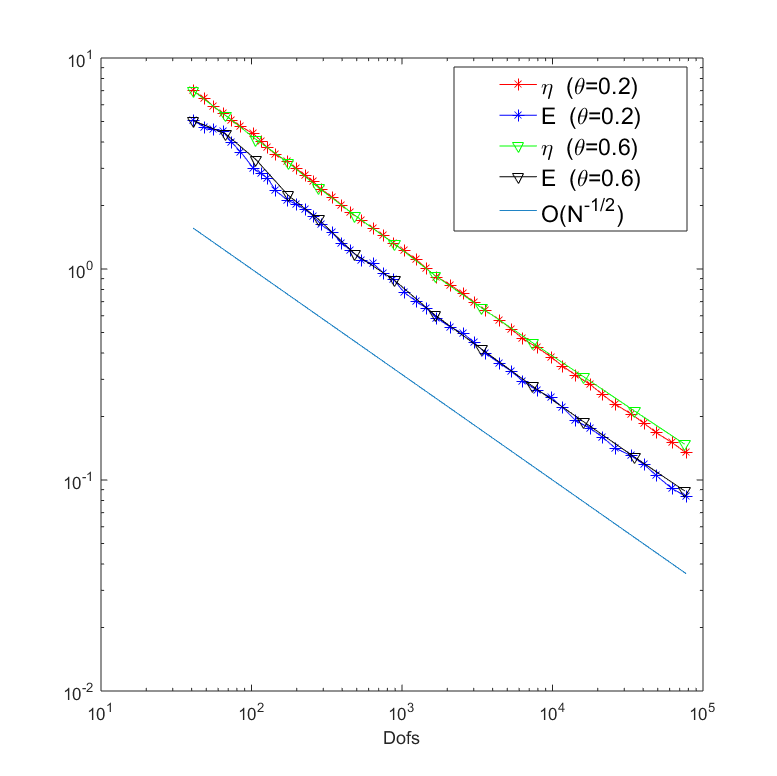}
\includegraphics[width=6.5cm,height=5.4cm,clip,trim=0cm 0cm 1cm 1cm]{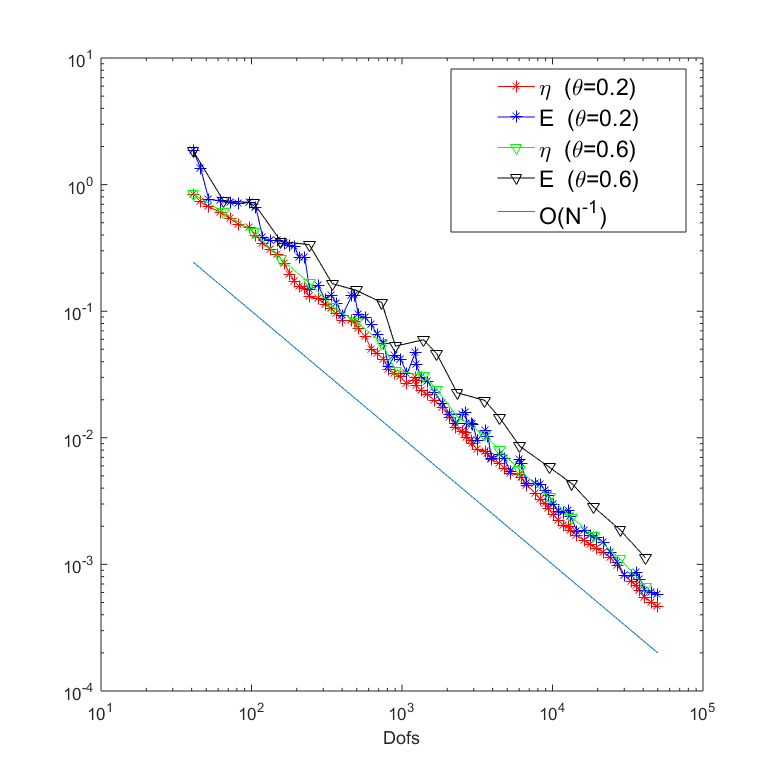}
\end{center}
\caption{\footnotesize Left: Convergence history for $k=1$. Right: Convergence history for $k=2$.}\label{ex1f1}
\end{figure}

\begin{figure}[htbp]
\begin{center}
\includegraphics[width=6.5cm,height=5.4cm,clip,trim=0cm 0cm 1cm 1cm]{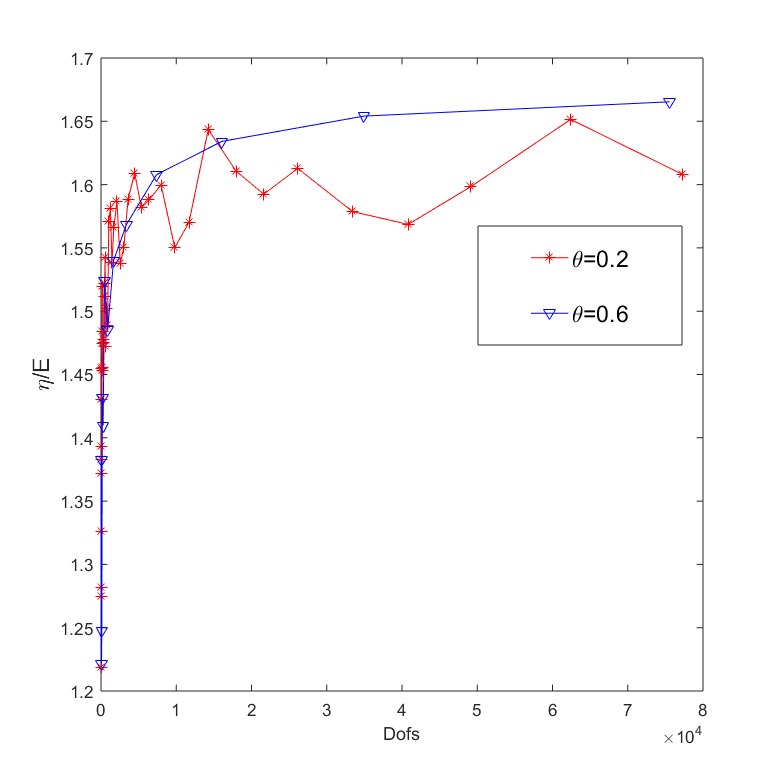}
\includegraphics[width=6.5cm,height=5.4cm,clip,trim=0cm 0cm 1cm 1cm]{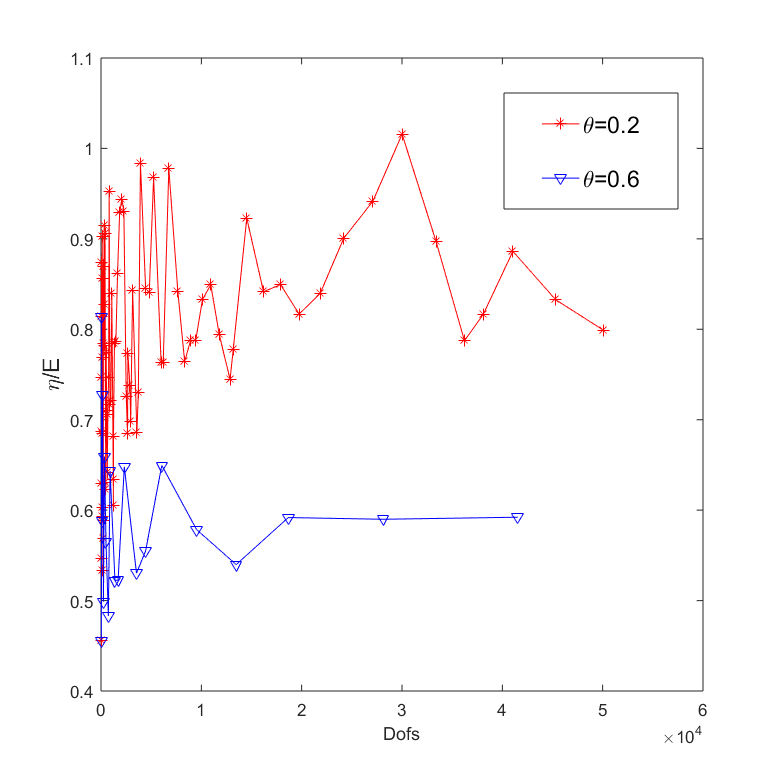}
\end{center}
\caption{\footnotesize Left: The effectiveness index for $k=1$. Right: The effectiveness index for $k=2$.}\label{ex1f2}
\end{figure}

\begin{figure}[htbp]
\begin{center}
\includegraphics[width=6.5cm,height=5.4cm,clip,trim=0cm 0cm 1cm 1cm]{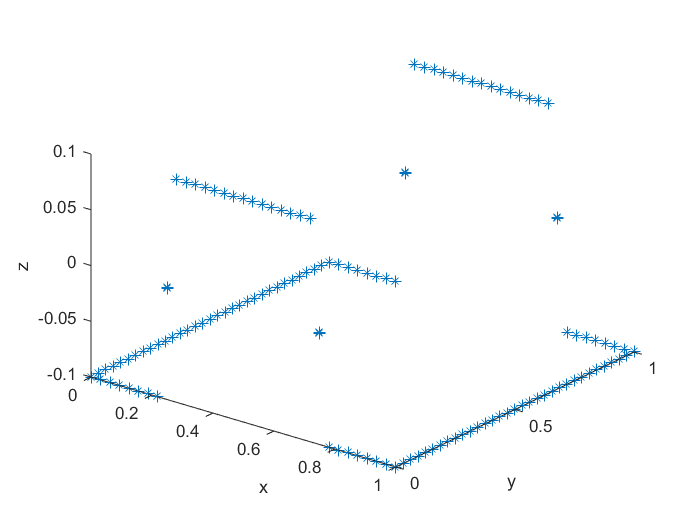}
\includegraphics[width=6.5cm,height=5.4cm,clip,trim=0cm 0cm 1cm 1cm]{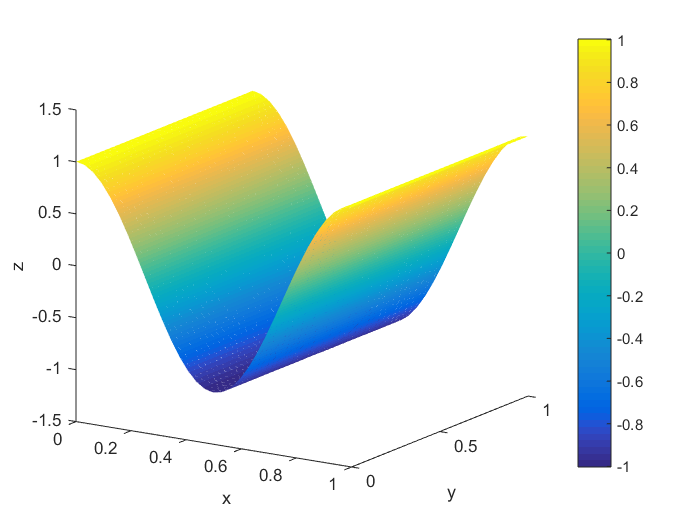}
\end{center}
\caption{\footnotesize Left: The profile of the numerical control for $k=1$. Right: The profile of the numerical adjoint state for $k=1$.}\label{ex1f3}
\end{figure}

We test the example for $k=1$ and $k=2$. From the convergence history in Figure \ref{ex1f1} for $\theta=0.2$ and $\theta=0.6$,
we find that the error $E$ is equivalent to the estimator $\eta$ and the error $E$ and the estimator $\eta$ can achieve the optimal
convergence order by adaptive refinement. Furthermore, the effectiveness index is presented in Figure \ref{ex1f2}, which indicates
the obtained a posteriori error estimator is very efficient. Finally, the profiles of the numerical control and adjoint state are
shown in Figure \ref{ex1f3}.

\begin{example}
We consider an example with a boundary term $\int_{\partial\Omega}yg_1dx$ in the objective functional. Then the adjoint problem possesses
the Nuemann boundary condition $\nabla z\cdot\textbf{n}=g_1$.
Here the designed domain is given by $\Omega=(-1,1)^2\backslash([0,1]\\
\times(-1,0])$. The control constraints and the regularization parameter are set as $u_a=-0.2$, $u_b=0.2$ and $\alpha=1$. Furthermore let the functions $f$, $y_d$ and $g$ be such that
the Neumann boundary control problem has the following exact solutions
\begin{align*}
y(r,\theta_1)&=0.\\
z(r,\theta_1)&=r^{2/3}\cos\Big(\frac{2}{3}\theta_1\Big),\\
u(r,\theta_1)&=\Pi_{U_{ad}}\Big(-r^{2/3}\cos\Big(\frac{2}{3}\theta_1\Big)\Big),
\end{align*}
where $r=\sqrt{x_1^2+x_2^2}$, $\theta_1=\arccos(r^{-1}x\cdot e_1)$ and $e_1=[1,0]^{T}$.
\end{example}

\begin{figure}[htbp]
\begin{center}
\includegraphics[width=6.5cm,height=5.0cm,clip,trim=1cm 0cm 1cm 1cm]{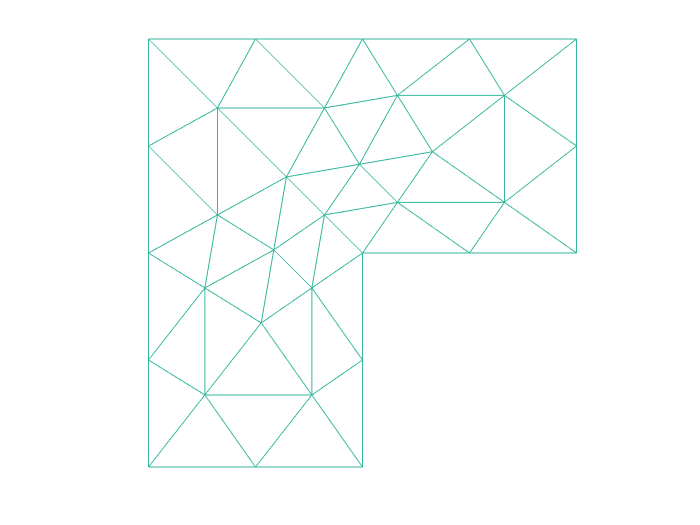}
\includegraphics[width=6.5cm,height=5.0cm,clip,trim=1cm 0cm 1cm 1cm]{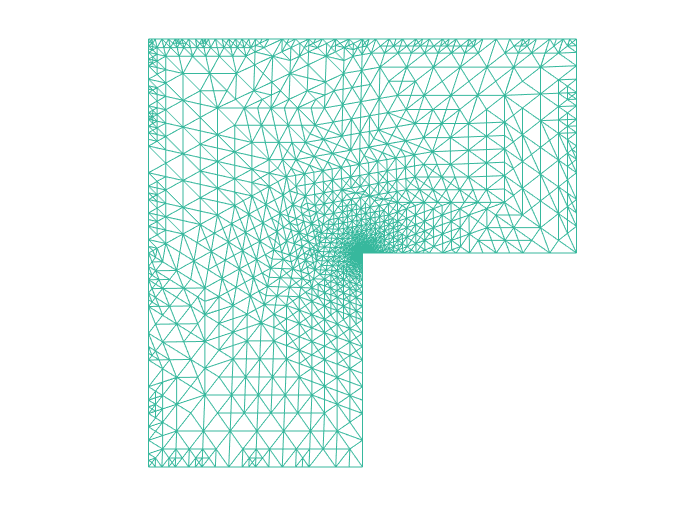}
\includegraphics[width=6.5cm,height=5.0cm,clip,trim=0cm 0cm 1cm 1cm]{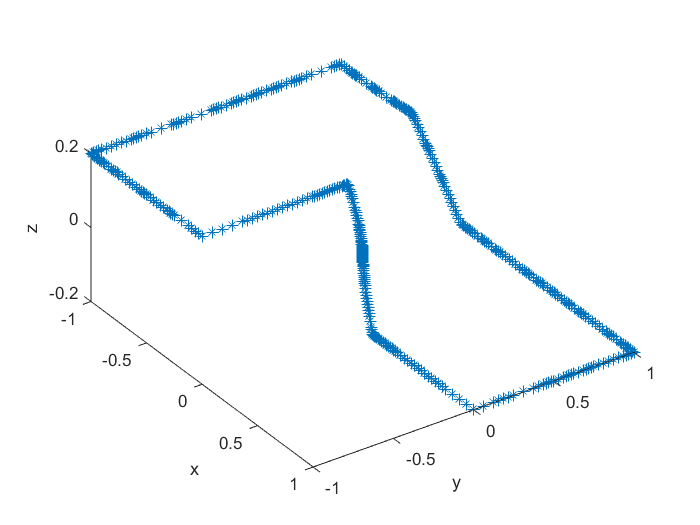}
\includegraphics[width=6.5cm,height=5.0cm,clip,trim=0cm 0cm 1cm 1cm]{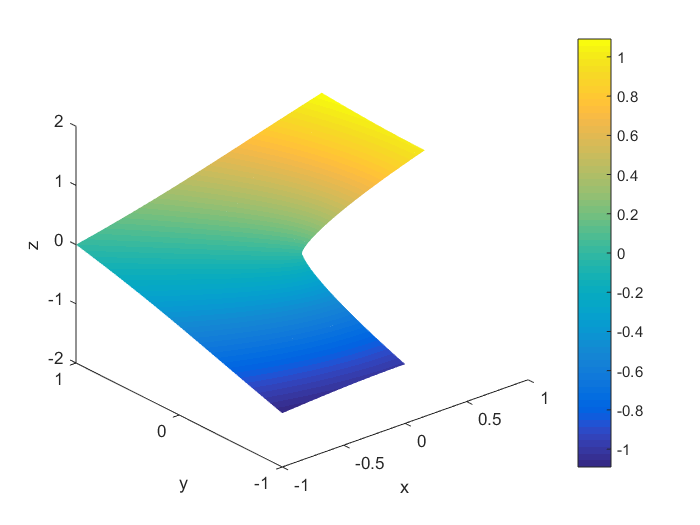}
\end{center}
\caption{\footnotesize The profiles of the initial mesh (Top-Left), the adaptive mesh (Top-Right), the numerical control (Bottom-Left)
and the numerical adjoint state (Bottom-Right) for $k=2$ and $\theta=0.4$.}\label{ex2f1}
\end{figure}

\begin{figure}[htbp]
\begin{center}
\includegraphics[width=6.5cm,height=5.4cm,clip,trim=0cm 0cm 1cm 1cm]{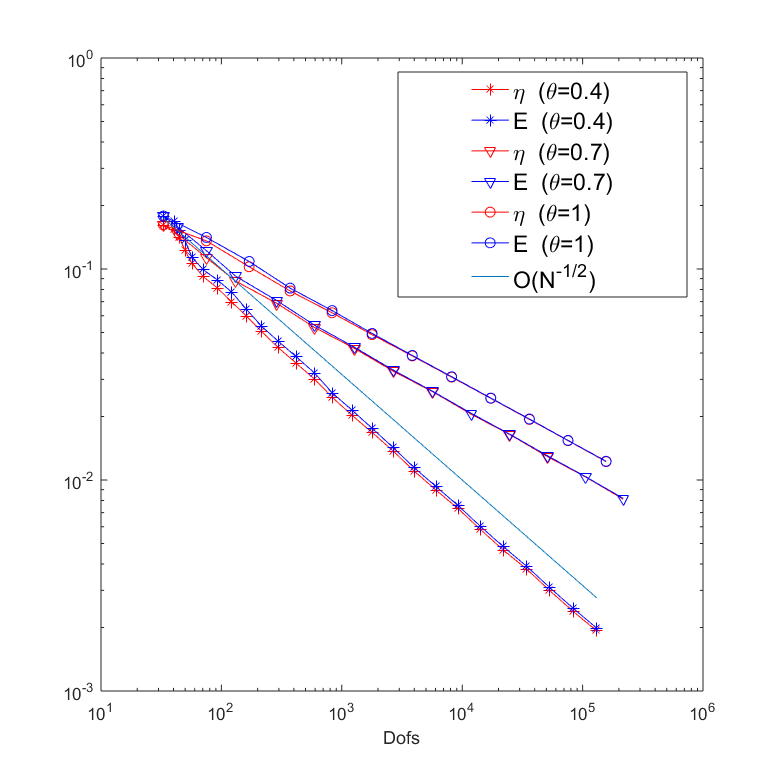}
\includegraphics[width=6.5cm,height=5.4cm,clip,trim=0cm 0cm 1cm 1cm]{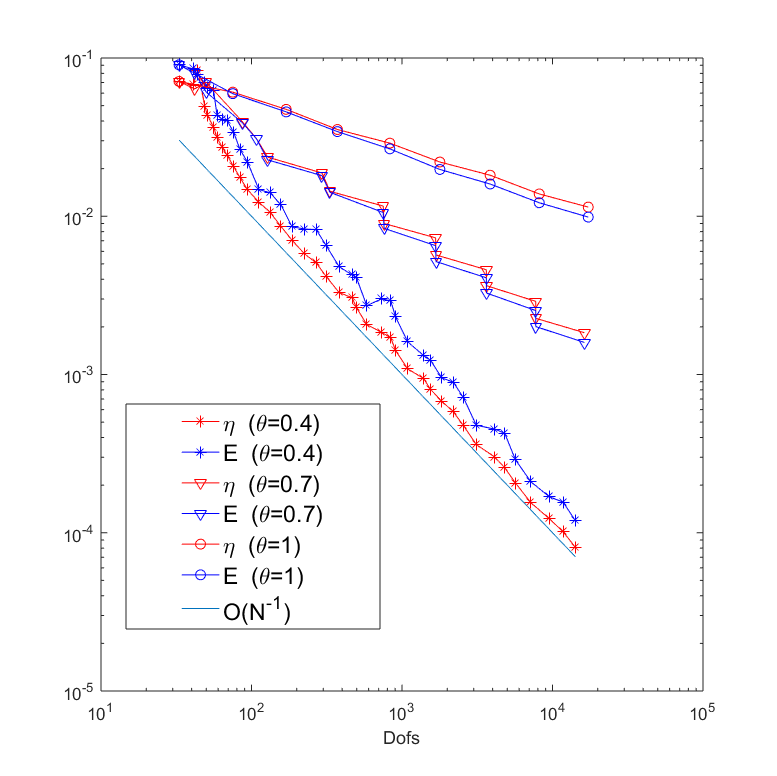}
\end{center}
\caption{\footnotesize The convergence history for $k=1$ (Left) and $k=2$ (Right).}\label{ex2f2}
\end{figure}

\begin{figure}[htbp]
\begin{center}
\includegraphics[width=6.5cm,height=5.4cm,clip,trim=0cm 0cm 1cm 1cm]{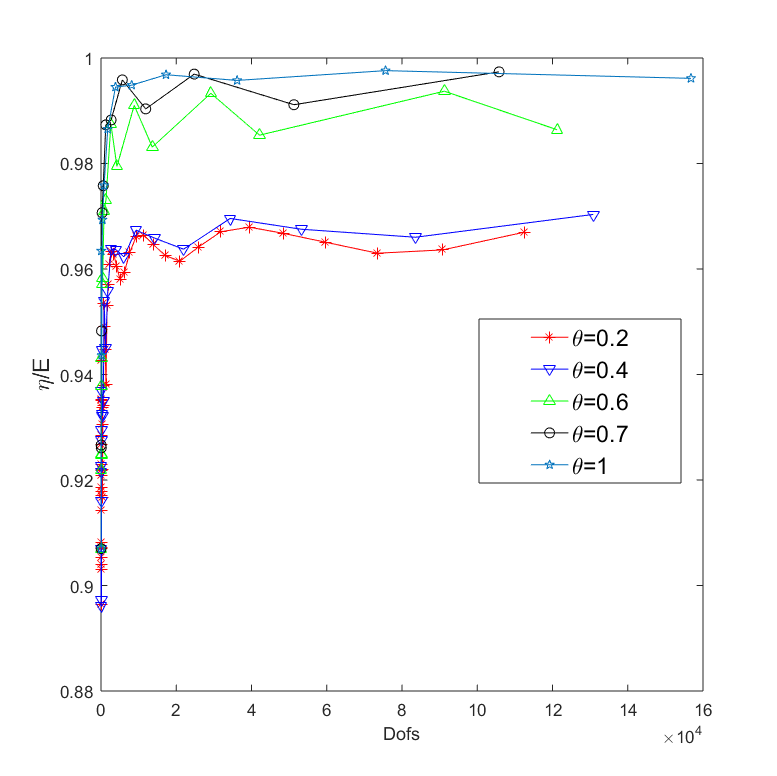}
\includegraphics[width=6.5cm,height=5.4cm,clip,trim=0cm 0cm 1cm 1cm]{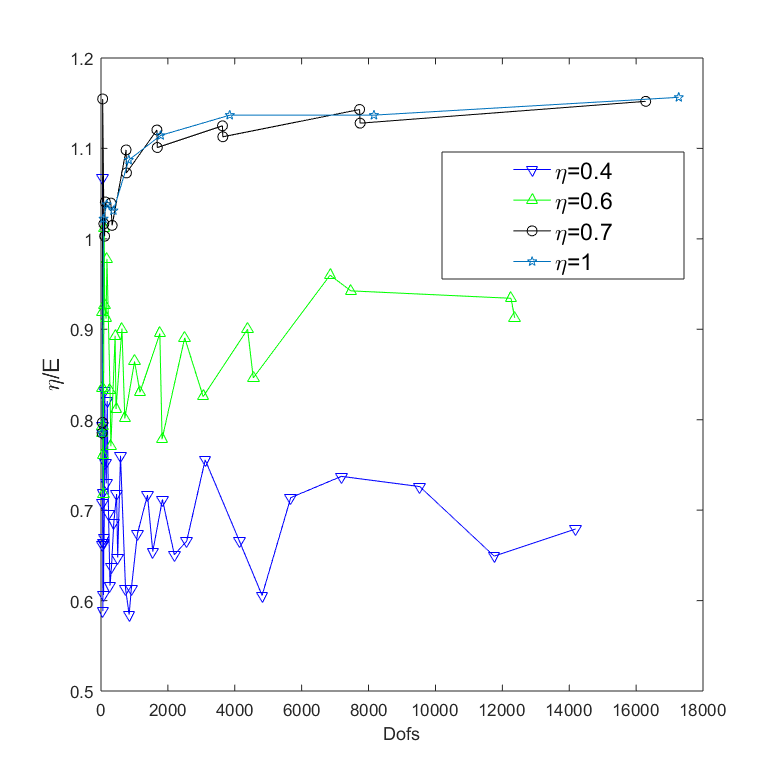}
\end{center}
\caption{\footnotesize The effectiveness index for $k=1$ (Left) and $k=2$ (Right).}\label{ex2f3}
\end{figure}

The adjoint exhibits a typical singularity at the reentrant corner of the domain $\Omega$. In Figure \ref{ex2f1}, we show
the profiles of the initial mesh, the adaptive mesh, the numerical control and the numerical adjoint state for $k=2$ and $\theta=0.4$.
We can find that the mesh nodes are concentrated around the reentrant corner where the singularity is induced.
Hence the obtained a posteriori error estimator can grab efficiently the singularity of the problem. In Figure \ref{ex2f2},
the convergence history for $k=1$ and $k=2$ is presented, which indicates that the estimator $\eta$ is equivalent to the error $E$ and
the estimator $\eta$ and the error $E$ can achieve the optimal convergence order while $\theta$ is less than a certain value.
In Figure \ref{ex2f3}, the effectiveness index for $k=1$ and $k=2$ are provided. We can find that the effectiveness index for $k=1$
is between 0.96 and 1 and the effectiveness index for $k=2$ is between 0.6 and 1.2, which means the obtained a posteriori error estimator
is very efficient.
\section{Conclusions}
\label{sec7}
In this paper, a Neumann boundary optimal control problem is considered. We use the hybridizable discontinuous Galerkin method
as the discretization technique, and the flux variables, the scalar variables and the boundary trace variables are approximated
by polynomials of degree $k$. Then an efficient and reliable a posteriori error estimator without any postprocessing solutions is obtained for the errors. Finally,
two numerical experiments are provided to verify the performance of the obtained a posteriori error estimator.

This work is just the first step for a posteriori error analysis of HDG methods for boundary control problems. Next we
extend the method and the result to the more complicated situations for instance the Dirichlet boundary control problem and
the Stokes optimal control problem.


\end{document}